\documentclass[12pt,a4paper]{article}

\usepackage{amsmath}
\usepackage{amsfonts}
\usepackage{epsfig,amstext,amssymb,amscd}
\usepackage{color, caption}
\usepackage{ntheorem}
\usepackage{graphicx}
\usepackage{vmargin}
\usepackage{subcaption}
\usepackage{tikz}
\usetikzlibrary{patterns}

\newtheorem{definition}{Definition}[section]
\newtheorem{proposition}[definition]{Proposition}
\newtheorem{theorem}[definition]{Theorem}
\newtheorem{lemma}[definition]{Lemma}
\newtheorem{remark}[definition]{Remark}
\newtheorem{corollary}[definition]{Corollary}
\newtheorem{example}[definition]{Example}

\newenvironment{proof}{{\rm Proof.}\noindent}{\hfill$\square$}

\setmarginsrb{2.5cm}{2cm}{2cm}{2cm}{14mm}{6mm}{0mm}{10mm}

\def\I{\mathcal{I}}

\def\M{\mathcal{M}}

\def\KK{\mathbb{K}}

\def\11{\mathbf{1}}

\newcommand{\suchthat}{\;\ifnum\currentgrouptype=16 \middle\fi|\;}

\makeatletter
\def\blfootnote{\xdef\@thefnmark{}\@footnotetext}
\makeatother

\def\beqa{\begin{eqnarray}}
\def\eeqa{\end{eqnarray}}

\begin{document}

\title{\bf Dendriform structures for restriction-deletion and restriction-contraction matroid Hopf algebras\let\thefootnote\relax\blfootnote{Date: \today}}
\author{Nguyen Hoang-Nghia\footnote{Nguyen.Hoang@lipn.univ-paris13.fr}, Christophe Tollu\footnote{Christophe.Tollu@lipn.univ-paris13.fr} and Adrian Tanasa\footnote{Adrian.Tanasa@lipn.univ-paris13.fr}}
\date{\today}
\maketitle

\begin{abstract}
\noindent
We endow the set of isomorphic classes of matroids with a new Hopf algebra structure, in which the coproduct is implemented via the combinatorial operations of restriction and deletion. We also initiate the investigation of dendriform coalgebra structures on matroids and introduce a monomial invariant which satisfy a convolution identity with respect to restriction and deletion.
\end{abstract}


\newpage

\section{Introduction}
\label{sect:introd}

It is widely acknowledged that major recent progress in combinatorics stems from the construction of algebraic structures associated to combinatorial objects, and from the design of algebraic invariants for those objects. For over three decades, numerous Hopf algebras with distinguished bases indexed by families of permutations, words, posets, graphs, tableaux, or variants thereof, have been brought to light (see, for example, \cite{novelli} and references within). The study of such combinatorial Hopf algebras (a class of free or cofree, connected, finitely graded bialgebras thoroughly characterized by Loday and Ronco \cite{LR98a}) has grown into an active research area; many connections with other mathematical domains and, perhaps more surprisingly, to theoretical physics (see, for example, \cite{io,io_bis} and references therein) have been uncovered and tightened. 

Since Schmitt's pioneering work \cite{s}, matroids have also been the subject of algebraic structural investigations, though to a much lesser extent than other familiar combinatorial species. In the present paper, we aim to carry the study of Hopf algebras on matroids one step forward by providing an alternative coproduct on matroids and by exploring their dendriform structures. 

Let us now outline the paper's contents. After a short reminder of the basic theory of matroids and a review of Schmitt's restriction-contraction Hopf algebra (section \ref{sec:mat-theo}), we define a new coproduct, relying on two standard operations on matroids, namely restriction and deletion, and show that the set of isomorphic classes of matroids can be endowed with a new commutative and cocommutative Hopf algebra structure, different from the one introduced by Schmitt (section \ref{sec:sel-del-Hopfalg}). We then prove that Schmitt's coproduct as well as ours can be adequately split into two pieces so as to give rise to two dendriform coalgebras (section \ref{sec:dendri-sel-del}). To the best of our knowledge, it is the first time that such an analysis has been carried out for matroids. Finally, we define a polynomial invariant of matroids which satisfy an identity which is the restriction-deletion analog of the more classical convolution identity satisfied by the Tutte polynomial for matroids (section \ref{sec:mat-pol}). Although that polynomial turns out to be a monomial, its definition exemplifies the usefulness of the theory of Hopf algebra characters in our context.

\section{Matroid theory; a restriction-contraction CHA}\label{sec:mat-theo}

\subsection{Matroid theory reminder}\label{subsec:mat-theo}

In this subsection we recall the definition and some basic properties of matroids (see, for example, 
J. Oxley's book \cite{Oxl92} or review article \cite{Oxl-rev}). 

\begin{definition}
A {\bf matroid} M is a pair $(E, \mathcal{I})$ consisting of a finite set $E$ and a collection of subsets of E satisfying the following set of axioms: $\mathcal{I}$ is non-empty, every subset of every member of $\mathcal{I}$ is also in $\mathcal{I}$ and, finally, if $X$ and $Y$ are in $\mathcal{I}$ and $|X| = |Y | + 1$, then there is an element $x$ in $X - Y$ such that $Y \cup \{x\}$ is in $\mathcal{I}$.
\end{definition}

One calls the set $E$ the {\bf ground set}. 
The elements of the set  $\mathcal{I}$ are the {\bf independent sets} of the matroid. 
A subset of $E$ that is not in $\mathcal{I}$ is called {\bf dependent}.

A particular class of matroids is the {\bf graphic matroids} (or {\bf cyclic matroids}), for whom the ground set is the set 
of edges of the graph and for whom the collection of independent sets is given by the sets of edges which do not contain all the edges of a cycle of the graph.

Let $E$ be an $n-$element set and let $\I$ be the collection of subsets 
of $E$ with at most $r$ elements, $0\le r\le n$.  The pair $(E,\I)$ 
is a matroid - the {\bf uniform matroid} $U_{r,n}$. The smallest (with respect to the cardinal of the edge set) non-graphic matroid is $U_{2,4}$.

%
The {\bf bases} of a matroid are the maximal independent sets of the respective matroid. 
Note that bases have all the same cardinality.
By relaxing this condition one then has delta-matroids \cite{dmatroids}.


Let $M=(E,\I)$ be a matroid and let $\cal B=\{ B\}$ be the collection of bases of $M$. 
Let ${\cal B}^\star = \{E - B: B \in {\cal B} \}$. 
Then ${\cal B}^\star$ is the collection of bases of a matroid $M^\star$ on E, 
the dual of $M$.

Let $M=(E,\I)$ be a matroid. 
The {\bf rank} $r(A)$ of $A \subset E$ is given by the following formula:
\begin{equation}\label{eq:rankfunc}
r(A) = max\{|B| \mbox{  s.t.  } B \in \I, B \subset A\}\ .
\end{equation}

\begin{lemma}[Lemma 1.3.1 \cite{Oxl92}]\label{lm:rankineq}The rank function $r$ of a matroid $M$ on a set $E$ satisfies the following condition: 
If $X$ and $Y$ are subsets of $E$, then \begin{equation}
r(X\cup Y) + r(X\cap Y) \leq r(X) + r(Y).
\end{equation}
\end{lemma}

\begin{lemma}[Proposition 1.3.5 \cite{Oxl92}]\label{lm:rankindepen}
Let $M=(E,\I)$ be a matroid with rank function $r$ and suppose that $X \subseteq E$. Then $X$ is independent if and only if $|X| = r(X)$.
\end{lemma}


Let $M=(E,\I)$ be a matroid.
The element $e \in E$ is a \textbf{loop} if and only if $\{e\}$ is a minimal dependent set of the matroid.
The element $e\in E$ is a \textbf{coloop} if and only if, for any basis $B$, $e\in B$.

Note that loops and coloops correspond, in graph theory, to bridges and self-loops, respectively (see Fig. \ref{fig:brigde} and \ref{fig:loop}).
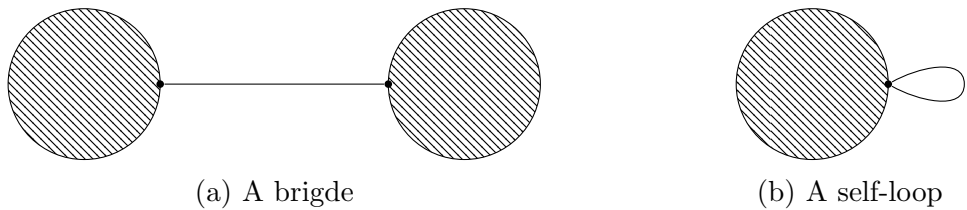
\begin{figure}[!ht]
\centering
\begin{subfigure}[b]{.45\linewidth}
\centering
\begin{tikzpicture}
\node (1) at (0,0) [inner sep=1pt,fill=black,circle] {};
\node (2) at (3,0) [inner sep=1pt,fill=black,circle] {};
\draw (1) to  (2);
\draw[pattern = north west lines] (-1,0) circle (1);
\draw[pattern = north west lines] (4,0) circle (1);
\end{tikzpicture}
\subcaption{A brigde}\label{fig:brigde}
\end{subfigure}
\begin{subfigure}[b]{.45\linewidth}
\centering
\begin{tikzpicture}
\node (1) at (0,0) [inner sep=1pt,fill=black,circle] {};
\draw (1) to [out=30,in=90] +(1,0) to [out=-90,in=-30] (1);
\draw[pattern = north west lines] (-1,0) circle (1);
\end{tikzpicture}
\subcaption{A self-loop}\label{fig:loop}
\end{subfigure}
\caption{Graphs with bridges and self-loops.}
\end{figure}

\medskip
Let $M$ be a matroid $(E,\I$) and $T$ be a subset of $E$. Let $\I|_T$ be the set $\{I\subseteq T \mbox{ s. t. } I\in\I\}$. 
The pair $(T,\I|_T)$ is a matroid, which is denoted by $M|_T$ - the {\bf restriction}  of $M$ to $T$.

Let $\I'=\{I\subseteq E-T: I \in \I\}$. 
The pair $(E-T,\I')$ is again a matroid. 
We denote this matroid by $M\backslash T$; we call this matroid the {\bf deletion} of $T$ from $M$. 

\begin{lemma}\label{lm:res-del}
Let $M$ be a matroid $(E,\I$) and $T$ be a subset of $E$. One has: 
\begin{equation}
M|_T = M\backslash_{E-T}.
\end{equation}
\end{lemma}







\subsection{A restriction-contraction matroid Hopf algebra}\label{subsec:sel-cont-Hopfalg}

In this subsection we recall the restriction-contraction matroid Hopf algebra introduced in \cite{s} (see also \cite{cs} for details).

Let us first give the following definition:

\begin{definition}
Let $M_1=(E_1,\I_1)$ and $M_2=(E_2,\I_2)$ be two matroids s. t. $E_1$ and $E_2$ are disjoint. Let 
$M_1 \oplus M_2 := (E_1 \cup E_2,\{I_1 \cup I_2: I_1 \in \I_1, I_2 \in \I_2\})$.
Then $M_1 \oplus M_2$ is a matroid - the {\bf direct sum} of $M_1$ and $M_2$.
\end{definition}

If a matroid $N$ is obtained 
from a matroid $M$ by any combination of restrictions and contractions  
or deletions, we call the matroid $N$ a {\bf minor} of $M$.
We write that a family of matroids is {\bf minor-closed} if 
it is closed under formation of minors and direct sums.
If $\mathcal{M}$ is a minor-closed family of matroids, we denote by $\widetilde{\mathcal{M}}$ 
the set of isomorphic classes of matroids 
belonging to $\mathcal{M}$. 
Direct sums induce a product on  $\widetilde{\mathcal{M}}$ (see \cite{s} for details). 
We denote by  $k(\widetilde{\mathcal{M}})$ the 
monoid algebra of $\widetilde{\mathcal{M}}$ over 
some commutative ring
$k$ with unit.


One has

\begin{proposition}
(Proposition $2.1$ of \cite{cs})
\label{prop-cs}
If $\mathcal{M}$ is a minor-closed family of matroids 
then $k(\widetilde{\mathcal{M}})$ is a coalgebra, 
with coproduct $\Delta$ and counit $\epsilon$ respectively
determined by 
\beqa
\label{defc}
\Delta^{(I)}(M) = \sum_{A\subseteq E} M|A \otimes M/A
\eeqa
and by 
$\epsilon(M) = \begin{cases} 1, \mbox{ if } E = \emptyset, \\ 0 \mbox{ otherwise ,} \end{cases}$ for all $M = (E,\mathcal{I}) \in \mathcal{M}$. If, furthermore, the family $\mathcal{M}$ is closed under formation of direct sums, then $k(\widetilde{\mathcal{M}})$ is a Hopf algebra, with product induced by direct sum.
\end{proposition}

In the rest of the paper, we follow \cite{cs}
and, by a slight abuse of notation, we denote in the same way a matroid and its isomorphic class,
since the distinction will be clear from the context (as it is already in 
Proposition \ref{prop-cs}). This is the same for the restriction-deletion matroid Hopf algebra that we will introduce in the following section.

The empty matroid (or $U_{0,0}$) is the unit of this Hopf algebra  and is denoted by $\mathbf{1}$. 

\begin{example}(Example 2.4 of \cite{cs})
Let $M = U_{k,n}$ be a uniform matroid with rank $k$. Its coproduct is given by
\[\Delta^{(I)}(U_{k,n}) = \sum_{i=0}^k \binom{n}{i} U_{i,i}\otimes U_{k-i,n-i} + \sum_{i=k+1}^n \binom{n}{i} U_{k,i}\otimes U_{0,n-i}\ .\]
\end{example}

\section{A restriction-deletion matroid Hopf algebra}\label{sec:sel-del-Hopfalg}


Let us define the following restriction-deletion map:

\begin{equation}\label{eq:coprod-sel-del}
\Delta^{(II)}:k(\widetilde{\mathcal{M}})\to k(\widetilde{\mathcal{M}})\otimes k(\widetilde{\mathcal{M}}),\ \
\Delta^{(II)} (M) := \sum_{A \subseteq E} M\mid A \otimes M \backslash A.
\end{equation}


\begin{example}
One has
\begin{itemize}
\item[1)] If $2k \leq n$, 
\begin{eqnarray}
\Delta^{(II)} (U_{k,n}) &=& \sum_{0 \leq i \leq k} {n \choose i} U_{i,i} \otimes U_{k,n-i} + \sum_{\substack{k < i \leq n \\ k \leq n-i}}{n \choose i} U_{k,i} \otimes U_{k,n-i} \cr && + \sum_{\substack{k < i \leq n \\ n-i < k}} {n \choose i}U_{k,i} \otimes U_{n-i,n-i}.
\end{eqnarray}
\item[2)] If $n < 2k$, 
\begin{eqnarray}
\Delta^{(II)} (U_{k,n}) &=& \sum_{\substack{0 \leq i \leq k\\ k \leq n-i}} {n \choose i}U_{i,i} \otimes U_{k,n-i} + \sum_{\substack{0 \leq i \leq k \\ n-i < k}} {n \choose i}U_{i,i} \otimes U_{n-i,n-i} \cr && + \sum_{k < i \leq n} {n \choose i}U_{k,i} \otimes U_{n-i,n-i}.
\end{eqnarray}
\end{itemize}
\end{example}

One has:

\begin{lemma}[Proposition 3.1.26 of \cite{Oxl92}]
\label{lm:coassoc}
Let $M=(E,\I)$ be a matroid.
\begin{itemize}
\item[1)] Let $X'$ be a subset of $X$ which is a subset of the ground set $E$. One has 
\begin{subequations}
\begin{align}
(M\mid X) \mid X' &= M\mid X', \label{eq:resres}\\
(M\mid X) \backslash X' &= M \mid (X -X'). \label{eq:resdel}
\end{align}
\end{subequations}
\item[2)] Let $X$ and $Y$ be subsets of $E$ such that $X$ and $Y$ are disjoint, one has 
\begin{subequations}
\begin{align}
(M\backslash X) \mid Y &= M \mid Y, \label{eq:delres}\\
(M\backslash X) \backslash Y &= M \backslash (X \cup Y). \label{eq:deldel}
\end{align}
\end{subequations}
\end{itemize}
\end{lemma}

\begin{proof}
These identities follow directly from the definitions of restriction and deletion for matroids (see previous section).
\end{proof}

\medskip
\begin{proposition}\label{prop:coass}
The coproduct in \eqref{eq:coprod-sel-del} is coassociative.
\end{proposition}

\begin{proof} Let $M=(E,\I)$ be a matroid. Let us calculate the left hand side (LHS) and right hand side (RHS) of the coassociativity identity. One has
\begin{eqnarray}\label{eq:lsh}
&&(\Delta^{(II)} \otimes Id) \circ \Delta^{(II)} (M) = \sum_{A\subseteq E} \Delta^{(II)}(M\mid A) \otimes M\backslash A \nonumber  \\
&&=\sum_{A\subseteq E} \left(\sum_{B \subseteq A} (M\mid A)\mid B \otimes (M\mid A)\backslash B \right) \otimes M\backslash A \cr
&& = \sum_{B \subseteq A\subseteq E} M\mid B \otimes M\mid (A-B) \otimes M\backslash A,
\end{eqnarray}
and
\begin{eqnarray}\label{eq:rhs}
(Id \otimes \Delta^{(II)}) \circ \Delta^{(II)} (M) &=& \sum_{A\subseteq E} M\mid A \otimes \Delta^{(II)}(M\backslash A) \cr
&=& \sum_{A\subseteq E} M\mid A \otimes \left( \sum_{C \in E-A} (M\backslash A)\mid C \otimes (M\backslash A) \backslash C \right)\cr
&=& \sum_{A\subseteq E} M\mid A \otimes \left( \sum_{C \in E-A} M\mid C \otimes M\backslash (A\cup C) \right)\cr
&=& \sum_{A\subseteq A\cup C \subseteq E} M\mid A \otimes   M\mid (A \cup C -A) \otimes M\backslash (A\cup C)\cr
&=& \sum_{C \subseteq A \subseteq E} M\mid C \otimes M\mid (A-C) \otimes M\backslash A.
\end{eqnarray}

Equations \eqref{eq:lsh} and \eqref{eq:rhs} lead to the conclusion.
\end{proof}

\medskip
%
Let us explicitly check the coassociativity of $\Delta^{(II)}$ on $U_{2,4}$.
\begin{eqnarray}
&& (\Delta^{(II)} \otimes Id) \circ \Delta^{(II)} (U_{2,4}) =  \11 \otimes \11 \otimes U_{2,4} + 4(U_{1,1} \otimes \11 + \11 \otimes U_{1,1} )\otimes U_{2,3} + 6(\11 \otimes U_{2,2} \cr
&& + 2 U_{1,1}\otimes U_{1,1} + U_{2,2} \otimes \11) \otimes U_{2,2} + 4(\11 \otimes U_{2,3} + 3 U_{1,1}\otimes U_{2,2} +3U_{2,2} \otimes U_{1,1} + U_{2,3}\otimes\11) \cr
&& \otimes U_{1,1} + (\11 \otimes U_{2,4} +4U_{1,1} \otimes U_{2,3} + 6 U_{2,2} \otimes U_{2,2}  + 4U_{2,3} \otimes U_{1,1} + U_{2,4} \otimes \11) \otimes \11 \cr
&& = \11 \otimes (\11 \otimes U_{2,4} +4U_{1,1} \otimes U_{2,3} + 6 U_{2,2} \otimes U_{2,2}  + 4U_{2,3} \otimes U_{1,1} + U_{2,4} \otimes \11) +     4U_{1,1}\otimes \cr
&& (\11 \otimes U_{2,3} + 3 U_{1,1}\otimes U_{2,2} +3U_{2,2} \otimes U_{1,1} + U_{2,3}\otimes\11) + 6U_{2,2} \otimes (\11 \otimes U_{2,2} + 2 U_{1,1}\otimes U_{1,1} \cr
&&+ U_{2,2} \otimes \11) + 4U_{2,3}\otimes (U_{1,1} \otimes \11 + \11 \otimes U_{1,1}) + U_{2,4}\otimes \11 \otimes \11  \cr
&&= (Id\otimes \Delta^{(II)}) \circ \Delta^{(II)} (U_{2,4}).
\end{eqnarray}

\medskip
\begin{proposition}\label{prop:cocommu}
The coproduct in \eqref{eq:coprod-sel-del} is cocommutative.
\end{proposition}
\begin{proof}
Let $\tau$ be a map $\widetilde{\M} \otimes \widetilde{\M} \longrightarrow \widetilde{\M} \otimes \widetilde{\M}$ defined by $M_1 \otimes M_2 \longmapsto M_2 \otimes M_1$. Using Lemma \ref{lm:res-del}, for $M=(E,\I) \in \widetilde{M}$, one has 
\begin{eqnarray}
\tau\circ \Delta^{(II)} (M) &=& \tau \left(\sum_{A\in E} M\mid A \otimes M \backslash A\right) = \sum_{A\in E} M \backslash A \otimes M\mid A 
= \sum_{A\in E} M \mid (E-A) \otimes M\backslash (E-A) \cr
&=& \Delta^{(II)} (M).
\end{eqnarray}
\end{proof}

\medskip
\begin{proposition}
$k(\widetilde{\M})$ is a cocommutative coalgebra with coproduct $\Delta^{(II)}$ and counit $\epsilon$ given by \begin{equation}\label{eq:counit}
\epsilon(M) = \begin{cases} 1 \; \mbox{ if } E=\emptyset \\ 0 \; \mbox{ otherwise,} \end{cases} \mbox{for all } M=(E,\I) \in \M.
\end{equation}
\end{proposition}

\begin{proof}
The proof follows directly from the definition \eqref{eq:counit}.
\end{proof}

\medskip
\begin{lemma}[Proposition 4.2.23 \cite{Oxl92}]\label{lm:comp}
Let $M_1$ and $M_2$ be two matroids. Let $A_1$ and $A_2$ be subset of $E_1$ and $E_2$, respectively. One then has
\begin{itemize}
\item[1)] \begin{equation}
M_1 \mid A_1 \oplus M_2 \mid A_2 = (M_1 \oplus M_2) \mid (A_1 \cup A_2).
\end{equation}
\item[2)] \begin{equation}
M_1 \backslash A_1 \oplus M_2 \backslash A_2 = (M_1 \oplus M_2) \backslash (A_1 \cup A_2).
\end{equation}
\end{itemize}
\end{lemma}

\begin{proof}
One can check these identities directly from the definitions of direct sum, restriction and deletion for matroids (see again the previous section). 
\end{proof}

\medskip
Let $\oplus^{\otimes 2}$ denote $(\oplus \otimes \oplus)\circ \tau_{23}$ where $\tau_{23}$ is the map $\widetilde{\M} \otimes \widetilde{\M} \otimes \widetilde{\M} \otimes \widetilde{\M} \longrightarrow \widetilde{\M} \otimes \widetilde{\M}\otimes \widetilde{\M} \otimes \widetilde{\M}$ defined by $M_1 \otimes M_2 \otimes M_3 \otimes M_4 \longmapsto M_1 \otimes M_3 \otimes M_2 \otimes M_4$.

\begin{proposition}
\label{prop:comp}
Let $M_1$ and $M_2$ be two matroids. One has 
\begin{equation}
\Delta^{(II)} (M_1 \oplus M_2) = \Delta^{(II)} (M_1) \oplus^{\otimes 2} \Delta^{(II)}(M_2).
\end{equation}
\end{proposition}

\begin{proof} Lemma \ref{lm:comp} leads to:
\begin{eqnarray}
\Delta^{(II)} (M_1 \oplus M_2) &=& \sum_{A \in E_1 \cup E_2}  M_1 \oplus M_2\mid A \otimes M_1 \oplus M_2\backslash A \cr
&=& \sum_{A_1 \in E_1, A_2 \cup E_2}  M_1 \oplus M_2\mid (A_1 \cup A_2) \otimes M_1 \oplus M_2\backslash (A_1  \cup A_2)\cr
&=& \sum_{A_1 \in E_1, A_2 \cup E_2}  (M_1\mid A_1 \oplus M_2\mid  A_2) \otimes (M_1\backslash A_1 \oplus M_2\backslash A_2)\cr
&=& \sum_{A_1 \in E_1, A_2 \cup E_2}  (M_1\mid A_1 \otimes M_1\backslash A_1) \oplus^{\otimes 2} (M_2\mid  A_2 \otimes  M_2\backslash A_2)\cr
&=& \left(\sum_{A_1 \in E_1}  M_1\mid A_1 \otimes M_1\backslash A_1 \right) \oplus^{\otimes 2} \left(\sum_{A_2 \cup E_2} M_2\mid  A_2 \otimes  M_2\backslash A_2\right)\cr
&=& \Delta^{(II)} (M_1) \oplus^{\otimes 2} \Delta^{(II)}(M_2),
\end{eqnarray}
which concludes the proof.
\end{proof}

Let us now explicitly check this identity on the matroids $M_1 = (\{1,2\}, \{\emptyset,\{1\},\{2\}\})$ and $M_2 = (\{3,4,5\},\{\emptyset,\{3\},\{4\},\{5\}\})$ ($U_{1,2}$ and respectively $U_{1,3}$).
One has: 
\begin{eqnarray}
\Delta^{(II)} (M_1) &=& \11 \otimes U_{1,2} + 2 U_{1,1} \otimes U_{1,1} +  U_{1,2} \otimes \11. \cr
\Delta^{(II)} (M_2) &=& \11 \otimes U_{1,3} + 3 U_{1,1} \otimes U_{1,2} + 3 U_{1,2} \otimes U_{1,1} + U_{1,3} \otimes \11.
\end{eqnarray}
which further leads to:
\begin{equation*}
M_1\oplus M_2 = (\{1,2,3,4,5\}, \{\emptyset,\{1\},\{2\},\{3\},\{4\},\{5\},\{1,3\},\{1,4\},\{1,5\},\{2,3\},\{2,4\},\{2,5\}\}).
\end{equation*}

Thus, one gets
\begin{eqnarray}
\Delta^{(II)}  (M_1\oplus M_2) &=& \11 \otimes (U_{1,2} \oplus U_{1,3}) + 2 U_{1,1} \otimes (U_{1,1} \oplus U_{1,3}) + 3 U_{1,1} \otimes (U_{1,2} \oplus U_{1,2}) \cr 
&& + U_{1,2} \otimes U_{1,3} + 6 U_{2,2} \otimes (U_{1,1} \oplus U_{1,2}) + 3 U_{1,2} \otimes (U_{1,2}\oplus U_{1,1}) \cr
&& + 3 (U_{1,2} \oplus U_{1,1}) \otimes U_{1,2} + 6 (U_{1,1} \oplus U_{1,2}) \otimes U_{1,1} + U_{1,3} \otimes U_{1,2} \cr 
&& + 3 (U_{1,2} \oplus U_{1,2}) \otimes U_{1,1} + 2 (U_{1,1} \oplus U_{1,3}) \otimes U_{1,1} + (U_{1,2}\oplus U_{1,3}) \otimes\11\cr
&=& \Delta^{(II)}(M_1)\oplus^{\otimes 2} \Delta^{(II)}(M_2).
\end{eqnarray}

\begin{proposition}\label{thm:bialg}
The triplet $(k(\widetilde{\M}),\oplus,\Delta^{(II)})$ is a commutative and cocommutative bialgebra.
\end{proposition}
\begin{proof}
The claim follows directly from Proposition \ref{prop:comp} and the results above.
\end{proof}


The main result of this section is:
\begin{theorem}\label{thm:Hopf}
The triplet $(k(\widetilde{\M}),\oplus,\Delta^{(II)})$ is a commutative and cocommutative Hopf algebra. The antipode $S$ of this Hopf algebra is given by \begin{equation}
\begin{cases}
S(\mathbf{1}) =\mathbf{1}, \\
S(M) = -M - \sum_{\emptyset \subsetneq A \subsetneq E} S(M\mid A) \oplus M\backslash A.
\end{cases}
\end{equation}
\end{theorem}
\begin{proof}
The bialgebra is graded by the cardinal of the ground set of matroids. Moreover, $\widetilde{\M}$ is connected, i. e. $\widetilde{\M}_0 = k \mathbf{1}$. This leads to the conclusion.
\end{proof}

\medskip

Let us end this section with the following example:

\begin{example}One has: 
\begin{equation}
S(U_{3,3}) = -U_{3,3}+3U_{1,1}\oplus U_{2,2} + 3U_{2,2}\oplus U_{1,1} -6 U_{1,1}\oplus U_{1,1} \oplus U_{1,1}.
\end{equation}
%
\end{example}

\section{Dendriform matroid coalgebras}
\label{sec:dendri-sel-del}

Let us first recall that a dendriform algebra \cite{Lod08, LR98a, LR98b, Foi07} is a family $(A,\prec,\succ)$ such that $A$ is a vector space and $\prec$, $\succ$ are two products on $A$, satisfying three axioms. Dually, one has a dendriform coalgebra $(C,\Delta_{\prec},\Delta_{\succ})$. 

\begin{definition} [Definition $2$ of \cite{Foi07}]
A {\bf dendriform coalgebra} is a family $(C,\Delta_{\prec},\Delta_{\succ})$ such that:
\begin{enumerate}
\item $C$ is a $k$-vector space and one has:
\begin{equation}
\Delta_{\prec} = \begin{cases}C \longrightarrow C \otimes C \\ a \longmapsto \Delta_{\prec} (a) = a'_{\prec} \otimes a''_{\prec}, \end{cases}
 \left| \ \  \Delta_{\succ} = \begin{cases}C \longrightarrow C \otimes C \\ a \longmapsto \Delta_{\succ} (a) = a'_{\succ} \otimes a''_{\succ}.\end{cases} \right.
\end{equation}
\item For all $a \in C$, one has:
\begin{eqnarray}
(\Delta_{\prec}\otimes Id)\circ \Delta_{\prec} (a) &=& (Id\otimes \Delta_{\prec} + Id \otimes \Delta_{\succ}) \circ \Delta_{\prec} (a),\label{eq:def-bidendriform1}\\
(\Delta_{\succ}\otimes Id)\circ \Delta_{\prec} (a) &=& (Id \otimes \Delta_{\prec})\circ \Delta_{\succ}(a),\label{eq:def-bidendriform2}\\
(\Delta_{\prec} \otimes Id + \Delta_{\succ}\otimes Id) \circ \Delta_{\succ} (a) &=& (Id \otimes \Delta_{\succ}) \circ \Delta_{\succ} (a)\label{eq:def-bidendriform3}.
\end{eqnarray}
\end{enumerate}
\end{definition}

If $C$ is a coalgebra, one defines 
\begin{equation}\label{eq:deltaStar}
\Delta_\ast: C \longrightarrow C \otimes C,\ \ a \longmapsto \Delta_\ast(a) = \Delta(a) - a \otimes 1 - 1 \otimes a
\end{equation}

\subsection{The restriction-deletion case}

Let us now define two maps on $\widetilde{\M}_+$ by:
\begin{equation}
\Delta^{(II)}_{\prec}(M) := \sum_{\substack{A \subsetneq E,\, A \neq \emptyset\\ |A| > r(A)}} M\mid A \otimes M\backslash A.
\end{equation}
\begin{equation}
\Delta^{(II)}_{\succ}(M) := \sum_{\substack{A \subsetneq E,\, A \neq \emptyset\\ |A| = r(A)}} M\mid A \otimes M\backslash A.
\end{equation}

If $A \subseteq E$, then $r(A) \leq |A|$. 
This directly leads to:
 \begin{equation}
\Delta^{(II)}_{\prec}(M) + \Delta^{(II)}_{\succ}(M) = \Delta^{(II)}_\ast(M).
\end{equation}

 From Lemma \ref{lm:rankindepen}, one can see that the coproduct $\Delta^{(II)}_\ast$ in Equation \eqref{eq:coprod-sel-del} is split into two parts: $\Delta^{(II)}_{\prec}$ sums on the dependent sets of the matroid and $\Delta^{(II)}_{\succ}$ sums on the independent sets of the matroid.

Let $k(\widetilde{\M})_+$ be the augmentation ideal of  $k(\widetilde{\M})$. 
We can now state the main result of this section:

\begin{proposition}\label{prop:dendri-sel-del}
The triplet $(k(\widetilde{\M})_+,\Delta^{(II)}_{\prec},\Delta^{(II)}_{\succ})$ is a dendriform coalgebra. 
\end{proposition}
\begin{proof}
Let $M$ be the matroid $(E,\I)$.
Let us first prove identity \eqref{eq:def-bidendriform1}. Its LHS writes:
\begin{eqnarray}\label{eq:lhs-bidendriform1}
(\Delta^{(II)}_{\prec}\otimes Id)\circ \Delta^{(II)}_{\prec} (M) &=& (\Delta^{(II)}_{\prec}\otimes Id) \left( \sum_{\substack{A \subsetneq E\\ |A| > r(A)}} M\mid A \otimes M\backslash A \right)
\cr
&
=
&
  \sum_{\substack{A \subsetneq E\\ |A| > r(A)}} \left(\sum_{\substack{B \subsetneq A\\ |B| > r(B)}} (M\mid A)\mid B \otimes (M\mid A)\backslash B\right) \otimes M\backslash A\cr
&=& \sum_{\substack{B \subsetneq A \subsetneq E\\ |A| > r(A)\\ |B|> r(B)}} M \mid B \otimes M\mid (A-B) \otimes M \backslash A.
\end{eqnarray}
On the other hand, the RHS of identity \eqref{eq:def-bidendriform1} can be rewritten as follows:
 \begin{eqnarray}\label{eq:rhs-bidendriform1}
&&(Id\otimes \Delta^{(II)}_{\prec} + Id \otimes \Delta^{(II)}_{\succ}) \circ \Delta^{(II)}_{\prec} (M) = (Id \otimes \Delta^{(II)}) \circ \Delta^{(II)}_{\prec} (M)\cr
&&= (Id \otimes \Delta^{(II)})\left( \sum_{X \subsetneq E, |X| > r(X)} M\mid X \otimes M\backslash X \right)\cr
&&= \sum_{\substack{X \subsetneq E\\ |X| > r(X)}} M\mid X \otimes \left(\sum_{Y \subsetneq E-X} (M\backslash X)\mid Y \otimes (M\backslash X)\backslash Y\right)\cr
&&= \sum_{\substack{X \subsetneq E \\ Y \subsetneq E-X \\ |X| > r(X)}} M \mid X \otimes M\mid Y \otimes M \backslash (X\cup Y).
\end{eqnarray}
From Lemma \ref{lm:rankineq}, one has: 
\begin{eqnarray}
r (X\cup Y) \leq  r(X) + r(Y) < |X|+|Y| = |X\cup Y|.
\label{eq:change-bidendriform1-2}
\end{eqnarray}
Setting  $X=B$ and $X \cup Y =A$ in equation \eqref{eq:rhs-bidendriform1} 
one now gets that identity \eqref{eq:def-bidendriform1} holds.

\smallskip
Let us now prove identity \eqref{eq:def-bidendriform2}. Its LHS writes: 
\begin{eqnarray}
(\Delta^{(II)}_{\succ}\otimes Id)\circ \Delta^{(II)}_{\prec} (M) &=& (\Delta^{(II)}_{\succ}\otimes Id)\left( \sum_{\substack{A \subsetneq E\\ |A| > r(A)}} M\mid A \otimes M\backslash A \right)\cr
&=& \sum_{\substack{A \subsetneq E\\ |A| > r(A)}} \left(\sum_{\substack{B \subsetneq A\\ |B| = r(B)}} (M\mid A)\mid B \otimes (M\mid A)\backslash B\right) \otimes M\backslash A\cr
&=& \sum_{\substack{B \subsetneq A \subsetneq E\\ |A| > r(A)\\ |B|= r(B)}} M \mid B \otimes M\mid (A-B) \otimes M \backslash A.
\end{eqnarray}
The RHS of identity \eqref{eq:def-bidendriform2} writes: 
\begin{eqnarray}
(Id \otimes \Delta^{(II)}_{\prec})\circ \Delta^{(II)}_{\succ}(M) &=& (Id \otimes \Delta^{(II)}_{\prec}) \left( \sum_{\substack{X \subsetneq E\\ |X| = r(X)}} M\mid X \otimes M\backslash X \right)\cr 
&=& \sum_{\substack{X \subsetneq E\\ |X| = r(X)}} M\mid X \otimes \left(\sum_{\substack{Y \subsetneq X\\ |Y| > r(Y)}}  (M\backslash X)\mid Y \otimes (M\backslash X) \backslash Y\right)\cr
&=& \sum_{\substack{Y \subsetneq X \subsetneq E\\ |X| = r(X)\\ |Y| > r(Y)}} M\mid X \otimes  M\mid Y \otimes M\backslash (X\cup Y).
\end{eqnarray}
As above, we can now set $X=B$ and $X\cup Y =A$ in the previous equation. One then concludes that identity \eqref{eq:def-bidendriform2} also holds.

Finally, identity \eqref{eq:def-bidendriform3} holds because of the coassociativity of $\Delta^{(II)}_\ast$ and since identities \eqref{eq:def-bidendriform1} and \eqref{eq:def-bidendriform2} also hold. 
This concludes the proof.
\end{proof}

We end this section by the following remark. The triplet 
$(k(\widetilde{\M})_+,\Delta^{(II)}_{\prec},\Delta^{(II)}_{\succ})$ is not a codendriform bialgebra. Indeed, the necessary compatibilities
for the dendriform coalgebra $(k(\widetilde{\M})_+,\Delta^{(II)}_{\prec},\Delta^{(II)}_{\succ})$ to be a codendriform bialgebra \cite{Foi07} 
write:
\begin{eqnarray}
\Delta^{(II)}_{\succ}(M N) &=& M'N'_{\succ} \otimes M''N''_{\succ} + M'\otimes M''N + MN'_{\succ} \otimes N''_{\succ} + N'_{\succ}\otimes MN''_{\succ} + M\otimes N,
\label{1}\\
\Delta^{(II)}_{\prec}(M N) &=& M'N'_{\prec} \otimes M''N''_{\prec} + M'N\otimes M'' + MN'_{\prec} \otimes N''_{\prec} + N'_{\prec}\otimes MN''_{\prec} + N\otimes M,
\end{eqnarray} where $\Delta^{(II)}_{\prec} (M) = M'_{\prec}\otimes M''_{\prec}$, $\Delta^{(II)}_{\prec} (N) = N'_{\prec}\otimes N''_{\prec}$, $\Delta^{(II)}_{\succ} (M) = M'_{\succ}\otimes M''_{\succ}$ and $\Delta^{(II)}_{\succ} (M) = M'_{\succ}\otimes M''_{\succ}$.
In our case, one gets the identity
\begin{eqnarray}
\Delta^{(II)}_{\succ}(M N) = N'_{\succ} \otimes MN''_\succ + M'_\succ \otimes M''_\succ N + M'_\succ N'_\succ \otimes M''_\succ N''_\succ ,
\end{eqnarray}
which is different of the identity \eqref{1} above.


\subsection{The restriction-contraction case}
\label{sec:dendri-sel-cont}

As in the case of the restriction-deletion coproduct analyzed in the previous subsection,
one can use the coproduct $\Delta^{(I)}$ to define a restriction-contraction dendriform coalgebra structure.

One defines two maps on $\widetilde{\M}_+$ by:
\begin{equation}
\Delta^{(I)}_{\prec}(M) := \sum_{\substack{A \subsetneq E,\, A \neq \emptyset\\ |A| > r(A)}} M\mid A \otimes M/ A.
\end{equation}
\begin{equation}
\Delta^{(I)}_{\succ}(M) := \sum_{\substack{A \subsetneq E,\, A \neq \emptyset\\ |A| = r(A)}} M\mid A \otimes M/ A.
\end{equation}

One has \begin{equation}
\Delta^{(I)}_{\prec}(M) + \Delta^{(I)}_{\succ}(M) = \sum_{\emptyset \neq A \subsetneq E} M\mid A \otimes M/ A = \Delta^{(I)}_\ast(M).
\end{equation}

\medskip
One has
\begin{proposition}
The triplet $(k(\widetilde{\M})_+,\Delta^{(I)}_{\prec},\Delta^{(I)}_{\succ})$ is a dendriform coalgebra. 
\end{proposition}
\begin{proof}
The proof of Proposition \ref{prop:dendri-sel-del} applies for this case as well.
\end{proof}

\medskip
This dendriform coalgebra is not a codendriform bialgebra for the same reasons as 
in the case of the dendriform restriction-deletion matroid coalgebra of the previous subsection.

\section{A matroid polynomial}\label{sec:mat-pol}

In this section we use an appropriate character of the restriction-deletion matroid Hopf algebra in order to define a certain matroid polynomial.





\begin{definition}
Let $M=(E,\I)$ be a matroid. Let $A\subseteq E$. One defines \begin{eqnarray}
c(A) &:=& \#\{e \in A \mid \{e\} \in \I\}, \\
l(A) &:=& \#\{e \in A \mid \{e\} \not\in \I\}.
\end{eqnarray}
\end{definition}

\begin{example}
\label{ex}
Let $M$ be a matroid on the ground set $\{1,2,3,4\}$ and the collection of independent sets be given by 
$\I = \{\emptyset,\{1\},\{2\},\{3\}\}$.
One has $l(E) = 3$ and $c(E)=1$.
Let $A=\{1,2,4\}$. One then has 
$l(A) = 2$ and $c(A)=1$.
\end{example}

\begin{remark}
For graphs, $l$ counts the number of self-loops and $c$ counts the number of edges which are not self-loops. 
\end{remark}

Note that the matroid of Example \ref{ex} is a graphic matroid, see Fig. \ref{fig:graphicMatroid}.
One has, as already noted above  $l(E) = 3$ and $c(E)=1$.

\begin{figure}[!ht]
\begin{center}
\begin{tikzpicture}
\node (1) at (0,0) [inner sep=1pt,circle,fill=black] {};
\node (2) at (2,0) [inner sep=1pt,circle,fill=black] {};
\draw (1) to [out=60,in=120]  node [above] {\tiny 1} (2)
		   (1) to [out=-60,in=-120] node [below] {\tiny 3}(2)
			(1) to node [above] {\tiny 2} (2)
	    	(2) to [out=30,in=90] +(1,0) node [right] {\tiny 4} to [out=-90,in=-30] (2);
\end{tikzpicture}
\caption{The graph corresponding to the matroid of Example \ref{ex}}\label{fig:graphicMatroid}
\end{center}
\end{figure}
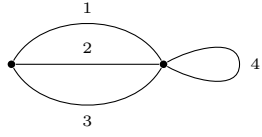

\begin{remark}\label{rm:infinitesimal}
\begin{itemize}
\item[1)] If $\{e\} \in \I$, then $M\mid e = U_{1,1}$.
\item[2)] If $\{e\} \not\in \I$, then $M\mid e = U_{0,1}$.
\end{itemize}
\end{remark}

Let $M=(E,\I)$ be a matroid.
Let us define the following polynomial:
\begin{equation}
\label{eq:poly}
P_M(x,y) := \sum_{A \subseteq E} (x-1)^{c(E)-c(A)} (y-1)^{l(A)}.
\end{equation}

\begin{remark}
Note that the definition above mimics the definition of the Tutte polynomial, where the role of the rank and of the nullity are played by the parameters $c$ and $l$, respectively.
\end{remark}

\begin{example}
One has:
\begin{equation}
P_{U_{2,4}} (x,y) = (x-1)^4 + 4(x-1)^3 + 6(x-1)^2 + 4(x-1) + 1 = x^4.
\end{equation}
\end{example}


As in \cite{AAM}, we now define: 

\begin{eqnarray}
\label{eq:def-dloop}
\delta_{\mathrm{loop}} (M) := \begin{cases}
1_\KK \mbox{ if } M 
= U_{0,1},\\
0_\KK \mbox{ otherwise},
\end{cases}
\end{eqnarray}

and

\begin{eqnarray}
\label{eq:def-dtree}
\delta_{\mathrm{coloop}} (M) := \begin{cases}
1_\KK \mbox{ if } M 
= U_{1,1},\\
0_\KK \mbox{ otherwise}.
\end{cases}
\end{eqnarray}
It is easy to check that these maps are 
{\it infinitesimal characters} of the restriction-deletion matroid Hopf algebra. 

\medskip
Following \cite{AAM} again, we define the map:
\begin{eqnarray}
\label{def-alpha}
\alpha(x,y,s,M) := \mbox{exp}_\ast s\{\delta_{\mathrm{coloop}}
+(y-1)\delta_{\mathrm{loop}}\}\ast \mbox{exp}_\ast s\{(x-1)\delta_{\mathrm{coloop}}+\delta_{\mathrm{loop}}\}
(M).
\end{eqnarray}
Using the definition of a Hopf algebra character
 one can directly check that the map \eqref{def-alpha} defined above is a character.


\medskip
Let us now show the relation between the map $\alpha$ 
and the polynomial in \eqref{eq:poly}.

\begin{lemma}\label{lm:exp} One has
\begin{equation}\label{eq:exp}
exp_\ast \{a\delta_{coloop} + b \delta_{loop}\} (M) = a^{c(M)}b^{l(M)}.
\end{equation}
\end{lemma}
\begin{proof}
The proof of Lemma $4.1$ of \cite{AAM} for the restriction-contraction  matroid Hopf algebra applies for the restriction-deletion matroid Hopf algebra as well (where one takes again into consideration that the role of the rank and of the nullity are played by the parameters $c$ and $l$).
\end{proof}

\begin{proposition}\label{prop:alpha}
One has
\begin{equation}
\alpha(x,y,s,M) = s^{|E|}P_M(x,y).
\end{equation} 
\end{proposition}
\begin{proof}
The proof of Proposition $4.3$ of \cite{AAM} for the restriction-contraction  matroid Hopf algebra applies for the restriction-deletion matroid Hopf algebra as well (where one takes again into consideration that the role of the rank and of the nullity are played by the parameters $c$ and $l$).
\end{proof}


 One further has:

\begin{corollary}
Let $M_1$ and $M_2$ be two matroids. One has \begin{equation}
P_{M_1 \oplus M_2} (x,y) = P_{M_1} (x,y) P_{M_2} (x,y).
\end{equation}
\end{corollary}
\begin{proof}
The conclusion follows directly from the definition of a Hopf algebra character and from Proposition \ref{prop:alpha} above.
\end{proof}


One then has:
\begin{eqnarray}
\label{eq:alpha}
\alpha(x,y,s,M) &=& 
\mbox{exp}_{\ast}
\left(s(\delta_{\mathrm{coloop}}+(y-1)\delta_{\mathrm{loop}})\right)\ast\mbox{exp}_{\ast}\left(s(-\delta_{\mathrm{coloop}}+\delta_{\mathrm{loop}})\right)
\nonumber\\
&\ast&\mbox{exp}_{\ast}\left(s(\delta_{\mathrm{coloop}}-\delta_{\mathrm{loop}})\right)\ast
\mbox{exp}_{\ast}\left(s((x-1)\delta_{\mathrm{coloop}}+\delta_{\mathrm{loop}})\right).
\end{eqnarray}


One has:
\begin{corollary}
 The polynomial in Equation \eqref{eq:poly} satisfies
\begin{equation}
 P_M(x,y)=\sum_{A\subset E} P_{M|A}(0,y) P_{M\backslash A}(x,0).
\end{equation}
\end{corollary}
\begin{proof}
The proof of Corollary $4.5$ of \cite{AAM} applies for the polynomial $P$.
\end{proof}
\begin{remark}
The identity above is the analog of a convolution identity for the Tutte polynomial proved initially in \cite{ELV} and \cite{KRS}. The difference comes from replacing the matroid contraction, in the Tutte polynomial case, with the matroid deletion, in the last factor of the identity.
\end{remark}

\medskip
Note that $P_M (x,y) \neq P_{M/e}(x,y) + P_{M\backslash e}(x,y)$ where $e$ is neither a loop nor a coloop. 

\medskip
Let us now give the recursive relations satisfied by the polynomial $P_M(x,y)$. 

\begin{proposition}
One has \begin{eqnarray}
P_M (x,y) &=& y P_{M\backslash e} (x,y) \mbox{ if } e \mbox{ is a loop,}\\
P_M (x,y) &=& x P_{M\backslash e} (x,y) \mbox{ otherwise.}
\end{eqnarray}
\end{proposition}
\begin{proof}
If $e$ is a loop, then $c(E-e) = c(E)$. One now has 
\begin{eqnarray}
P_M(x,y) &=& \sum_{\substack{A \subseteq E\\e\in A}} (x-1)^{c(E)-c(A)}(y-1)^{l(A)} + \sum_{\substack{A \subseteq E\\e\not\in A}} (x-1)^{c(E)-c(A)}(y-1)^{l(A)}\cr
&=& \sum_{\substack{A' \subseteq E-e}} (x-1)^{c(E-e)-c(A')}(y-1)^{l(A')+1} + \sum_{\substack{A \subseteq E-e}} (x-1)^{c(E)-c(A)}(y-1)^{l(A)}\cr
&=& y\sum_{\substack{A \subseteq E-e}} (x-1)^{c(E-e)-c(A)}(y-1)^{l(A)}\cr
&=& yP_{M \backslash e} (x,y).
\end{eqnarray}

Similarly, if $e$ is not a loop, one then has $P_M (x,y) = x P_{M\backslash e} (x,y)$.
\end{proof}

\medskip
\begin{corollary}
One has 
\begin{equation}
P_M(x,y) = x^{c(E)}y^{l(E)}.
\end{equation}
\end{corollary}

\medskip
Finally, one notices that
$P_M(x,y) \neq P_{M^\ast} (x,y),$
and $P_M(x,y) \neq P_{M^\ast} (y,x).$
This follows, for example, from analyzing the cases of the matroids $U_{0,1} = U_{1,1}^\ast$ and $U_{1,2} = U_{1,2}^\ast$.

%
%
%
%

\bigskip

\noindent
{\bf Acknowledgements:} The authors kindly acknowledge G. Duchamp and G. Koshevoy for discussions on positroids, discussions which have eventually led to our work on matroids. A. Tanasa is partially funded by the grants ANR JCJC “CombPhysMat2Tens” and PN 09 37 01 02.

{\small
\noindent
Nguyen Hoang-Nghia and Christophe Tollu\\
LIPN, Institut Galil\'ee, CNRS UMR 7030, \\
Univ. Paris 13, Sorbonne Paris Cit\'e,\\
99 av. J.-B. Clement, 93430 Villetaneuse, France, EU

\medskip

\noindent
Adrian Tanasa\\
LIPN, Institut Galil\'ee, CNRS UMR 7030, \\
Univ. Paris 13, Sorbonne Paris Cit\'e,\\
99 av. J.-B. Clement, 93430 Villetaneuse, France, EU\\
Horia Hulubei National Institute for Physics and Nuclear Engineering,\\ 
P.O.B. MG-6,
077125 Magurele, Romania, EU}

\end{document}